\documentclass[12pt]{amsart}
\usepackage[utf8]{inputenc}
\usepackage{commath}
\usepackage[ maxnames=4,maxalphanames=4,style=alphabetic-verb,doi=false,
  url=false,
  isbn=false,
  eprint=false
]{biblatex}
\bibliography{main.bib}
\usepackage[margin=1in]{geometry}
\usepackage{graphicx}
\usepackage{caption}
\usepackage{subcaption}
\usepackage{amsmath,amsfonts,amssymb,amsthm,amsaddr,etoolbox}
\usepackage{latexsym,hyperref}
\usepackage[most]{tcolorbox}
\usepackage{xcolor}
\usepackage{scrextend}
\usepackage{quiver}
\usepackage{enumitem}
\usepackage[normalem]{ulem}
\usepackage{mathtools}
\usepackage[linewidth=1pt]{mdframed}
\usepackage{tikz-cd}
\usepackage{float}

\DeclareFieldFormat{postnote}{#1}
\DeclareFieldFormat{multipostnote}{#1}

\author[R. Easwar]{Rohun Easwar}

\email{rohuneaswar@gmail.com}
\title[Proofs of some conjectures on binomial coefficients]{An application of Jacobi's residue formula and proofs of some other conjectures on binomial coefficients}
\keywords{Jacobi's residue formula, MacMahon's master theorem, binomial coefficients, generating functions, analytic combinatorics}
\subjclass[2020]{05A10, 05A15, 11B65}

\newcommand\restr[2]{{
  \left.\kern-\nulldelimiterspace 
  #1 
  \littletaller 
  \right|_{#2} 
  }}

\theoremstyle{plain}
\newtheorem{defn}{Definition}[section]

\newtheorem{lem}[defn]{Lemma}

\newtheorem*{prop*}{Proposition}
\newtheorem*{thm*}{Theorem}
\newtheorem{thm}[defn]{Theorem}
\newtheorem{cor}[defn]{Corollary}

\newtheorem*{claim*}{Claim}

\theoremstyle{remark}
\newtheorem{rem}[defn]{Remark}
\theoremstyle{remark}

\theoremstyle{remark}

\theoremstyle{remark}

\theoremstyle{remark}

\theoremstyle{remark}

\theoremstyle{remark}

\theoremstyle{remark}

\theoremstyle{remark}

\numberwithin{equation}{section}

\theoremstyle{definition}

\makeatletter
\let\save@mathaccent\mathaccent
\newcommand*\if@single[3]{%
  \setbox0\hbox{${\mathaccent"0362{#1}}^H$}%
  \setbox2\hbox{${\mathaccent"0362{\kern0pt#1}}^H$}%
  \ifdim\ht0=\ht2 #3\else #2\fi
  }
\newcommand*\rel@kern[1]{\kern#1\dimexpr\macc@kerna}
\newcommand*\widebar[1]{\@ifnextchar^{{\wide@bar{#1}{0}}}{\wide@bar{#1}{1}}}
\newcommand*\wide@bar[2]{\if@single{#1}{\wide@bar@{#1}{#2}{1}}{\wide@bar@{#1}{#2}{2}}}
\newcommand*\wide@bar@[3]{%
  \begingroup
  \def\mathaccent##1##2{%
    \let\mathaccent\save@mathaccent
    \if#32 \let\macc@nucleus\first@char \fi
    \setbox\z@\hbox{$\macc@style{\macc@nucleus}_{}$}%
    \setbox\tw@\hbox{$\macc@style{\macc@nucleus}{}_{}$}%
    \dimen@\wd\tw@
    \advance\dimen@-\wd\z@
    \divide\dimen@ 3
    \@tempdima\wd\tw@
    \advance\@tempdima-\scriptspace
    \divide\@tempdima 10
    \advance\dimen@-\@tempdima
    \ifdim\dimen@>\z@ \dimen@0pt\fi
    \rel@kern{0.6}\kern-\dimen@
    \if#31
      \overline{\rel@kern{-0.6}\kern\dimen@\macc@nucleus\rel@kern{0.4}\kern\dimen@}%
      \advance\dimen@0.4\dimexpr\macc@kerna
      \let\final@kern#2%
      \ifdim\dimen@<\z@ \let\final@kern1\fi
      \if\final@kern1 \kern-\dimen@\fi
    \else
      \overline{\rel@kern{-0.6}\kern\dimen@#1}%
    \fi
  }%
  \macc@depth\@ne
  \let\math@bgroup\@empty \let\math@egroup\macc@set@skewchar
  \mathsurround\z@ \frozen@everymath{\mathgroup\macc@group\relax}%
  \macc@set@skewchar\relax
  \let\mathaccentV\macc@nested@a
  \if#31
    \macc@nested@a\relax111{#1}%
  \else
    \def\gobble@till@marker##1\endmarker{}%
    \futurelet\first@char\gobble@till@marker#1\endmarker
    \ifcat\noexpand\first@char A\else
      \def\first@char{}%
    \fi
    \macc@nested@a\relax111{\first@char}%
  \fi
  \endgroup
}
\makeatother

\begin{document}
\begin{abstract}
We prove seven conjectures involving binomial coefficients. These conjectures appear in The On-Line Encyclopedia of Integer Sequences (OEIS). Among these is an application of Jacobi's residue formula. We also apply MacMahon's master theorem to give a new and independently obtained proof of a recently established result.
\end{abstract}
\maketitle






















\section{Introduction}
The  On-Line Encyclopedia of Integer Sequences (OEIS) \cite{OEIS} is a database that records sequences of integers, and entries often contain comments on the sequence featured in it. These comments are sometimes conjectures about the properties of the sequence. In this paper, we prove seven such conjectures. We also present a new proof, independently obtained, of a recently established result.

In \S~\ref{sec 2}, we discuss some notation and preliminaries that will be useful throughout the paper. In \S~\ref{sec 3}, we derive an expression for the small 3-Schr\"{o}der numbers $\mathrm{A034015}(n)$, found in \cite{A034015}. The evaluation of the diagonal of a rational function in terms of $\mathrm{A002897}(n)$ (\cite{A002897}) and an application of MacMahon's master theorem (see \cite{MMTshort}) are seen in \S~\ref{sec 4}. Jacobi's residue formula (see \cite{GesselResidue}) is used in \S~\ref{sec 5} to express $\mathrm{A002897}(n)$ (\cite{A002897}) as the coefficient of a multivariate polynomial. In \S~\ref{sec 6}, the ordinary power series generating function of $\mathrm{A290575}(n)$ (\cite{A290575}) is shown to have a square root with integer coefficients. The next section, \S~\ref{sec 7}, is concerned with recurrence relations satisfied by $\mathrm{A127361}(n)$ (\cite{A127361}) and $\mathrm{A105872}(n)$ (\cite{A105872}). Finally, \S~\ref{sec 8} discusses a generalisation of an asymptotic relation satisfied by $\mathrm{A094213}(n)$ (\cite{A094213}).

\subsection*{Acknowledgements}
The author thanks Amit Kuber for reading the manuscript and providing helpful comments on its presentation and organisation. The author also thanks Rahul Joy Das for valuable discussions.

\section{Notation and Preliminaries}\label{sec 2}
Denote by $\mathbb{C}((x_1,\dots,x_m))$ the formal Laurent series ring in indeterminates $x_1,\dots,x_m$ over $\mathbb{C}$, that is, the ring of formal series
\[\sum\limits_{k_1\ge n_1} \cdots \sum\limits_{k_m\ge n_m} a_{k_1,\dots,k_m} x_1^{k_1}\cdots x_m^{k_m},\]
where $n_1,\dots,n_m\in\mathbb{Z}$ and $a_{k_1,\dots,k_m}\in\mathbb{C}$.

Denote by $\mathbb{C}[[x_1,\dots,x_m]]$ the formal power series ring, and by $\mathbb{C}(x_1,\dots,x_m)$ the field of rational functions in indeterminates $x_1,\dots,x_m$ over $\mathbb{C}$.

\begin{rem}
In this paper, we use the same symbol to refer to a power series with positive radius of convergence and the analytic function represented by it on its disc of convergence. Also note that formal and analytic differentiation agree in the disc of convergence (see \cite[2.1]{ACSV}).
\end{rem}

For $\alpha_1,\dots,\alpha_m\in \mathbb{Z}$, define the corresponding \emph{coefficient extraction} operator 
\[[x_1^{\alpha_1}\cdots x_m^{\alpha_m}]:\mathbb{C}((x_1,\dots,x_m))\to\mathbb{C} \text{ by}\] 
\[\sum\limits_{k_1\ge n_1} \cdots \sum\limits_{k_m\ge n_m} a_{k_1,\dots,k_m} x_1^{k_1}\cdots x_m^{k_m} \mapsto a_{\alpha_1,\dots,\alpha_m}.\]

For $\alpha_1,\dots,\alpha_m\in \mathbb{Z}_{\ge 0}$, we use the same notation for the restriction of the above-defined coefficient extraction operators to $\mathbb{C}[[x_1,\dots,x_m]]$.

\section{An expression for the small 3-Schr\"{o}der numbers}\label{sec 3}

In this section, we derive an expression for the small 3-Schr\"{o}der numbers. We begin by recalling definitions pertaining to $m$-Schr\"{o}der paths, $m$-Schr\"{o}der numbers, and their weighted versions. The reader may use \cite{YangJiang} as a reference.

\begin{defn}\cite[Definition~2.2]{YangJiang}
Let $m\ge 2$ and $n\ge 1$ be positive integers.

An \emph{$m$-Schr\"{o}der path of length $mn$} is a lattice path from $(0,0)$ to $(mn, 0)$ which never goes below the x-axis, with steps $U = (1, 1)$, $D^{(m)} = (1, 1 - m)$, and $H^{(m)} = (2, 2 - m)$.

The \emph{$n^{\text{th}}$ $m$-Schr\"{o}der number}, denoted $r_n^{(m)}$, is the number of $m$-Schr\"{o}der paths of length $mn$.

An $m$-Schr\"{o}der path with no steps of type $H^{(m)}$ ending on the $x$-axis is called a \emph{small $m$-Schr\"{o}der path}.

The \emph{$n^{\text{th}}$ small $m$-Schr\"{o}der number}, denoted $s_n^{(m)}$, is the number of small $m$-Schr\"{o}der paths of length $mn$.
\end{defn}

We now discuss the weighted versions of the above definitions.

\begin{defn}\cite{YangJiang}
Let $m\ge 2$ and $n\ge 1$ be positive integers and $a,b\in\mathbb{C}$.

An \emph{$(m;a,b)$-Schr\"{o}der path of length $mn$} is an $m$-Schr\"{o}der path of length $mn$ whose each step of the form $D^{(m)}$ is assigned weight $a$ and each step of the form $H^{(m)}$ is assigned weight $b$.

The \emph{weight} of an $(m;a,b)$-Schr\"{o}der path is the product of the weights of its steps.

Let $r_n^{(m,a,b)}$ denote the sum of the weights of all $(m;a,b)$-Schr\"{o}der paths of length $mn$. We call this the \emph{$n^{\text{th}}$ $(m;a,b)$-Schr\"{o}der number}.

The \emph{$n^{\text{th}}$ small $(m;a,b)$-Schr\"{o}der number}, denoted $s_n^{(m,a,b)}$, is the sum of the weights of all small $(m;a,b)$-Schr\"{o}der paths of length $mn$.
\end{defn}

\begin{rem}\label{rem:unweighting}
From the above definition, it is clear that $r_n^{(m)}=r_n^{(m,1,1)}$ and $s_n^{(m)}=s_n^{(m,1,1)}$.
\end{rem}

The sequence A034015(\cite{A034015}) of small $3$-Schr\"{o}der numbers is given by 
\[s_{n+1}^{(3)}=\mathrm{A034015}(n)= \sum\limits_{k=0}^n\frac{2^k\binom{n}{k}\binom{2n+2}{k}}{k+1}
\]
for $n\ge 0$.

\begin{thm}\cite[Theorem~2.4]{YangJiang}\label{thm:Yang2.4}
Let $m\ge 2$ and $n\ge 1$ be positive integers and $a,b\in\mathbb{C}$. Then
\[r_n^{(m,a,b)}=\frac{1}{n}\sum\limits_{j\ge 1}\binom{(m-1)n}{j-1}\binom{n}{j}\left(a+b\right)^ja^{n-j}.
\]
\end{thm}

\begin{thm}\cite[Theorem~2.9]{YangJiang}\label{thm:Yang2.9}
Let $m\ge 2$ and $n\ge 1$ be positive integers and $a,b\in\mathbb{C}$. Then $(a+b)s_n^{(m,a,b)}=ar_n^{(m,a,b)}$. In particular, $2s_n^{(m)}=r_n^{(m)}$.
\end{thm}

\begin{rem}
A comment by Sloane in \cite{A034015} states that theorems \ref{thm:Yang2.4} and \ref{thm:Yang2.9} may be used to obtain a formula for $s_{n+1}^{(3)}$. Subsequently, Yan explicitly stated the identity $s_{n+1}^{(3)}=\frac{1}{n+1}\sum\limits_{r=1}^{n+1}\binom{2n+2}{r-1}\binom{n+1}{r}2^{r-1}$ in \cite{A034015}, referencing \cite[Theorem~4.4]{Parking}. We provide a proof of this identity using Sloane's comment.
\end{rem}

\begin{lem}\label{lem:intermediate-id}
For $n\ge 0$, we have
\[s_{n+1}^{(3)} = \frac{1}{n+1}\sum\limits_{r=1}^{n+1}\binom{2n+2}{r-1}\binom{n+1}{r}2^{r-1}.
\]
\end{lem}
\begin{proof}
We have
\begin{align*}
s_{n+1}^{(3)} &= \frac{1}{2}r_{n+1}^{(3)} \text{ [By Theorem \ref{thm:Yang2.9}]}\\
&= \frac{1}{2}r_{n+1}^{(3,1,1)} \text{ [By Remark \ref{rem:unweighting}]}\\
&= \frac{1}{2(n+1)}\sum\limits_{j\ge 1}\binom{2(n+1)}{j-1}\binom{n+1}{j}2^j \text{ [By Theorem \ref{thm:Yang2.4}]}\\
&= \frac{1}{n+1}\sum\limits_{r=1}^{n+1}\binom{2n+2}{r-1}\binom{n+1}{r}2^{r-1}. \qedhere
\end{align*}
\end{proof}

In \cite{A034015}, Weiner conjectured the following result, which we prove.

\begin{thm}
For $n\ge 0$, we have
\[s_{n+1}^{(3)}=\frac{1}{n+1}\sum\limits_{i=0}^n\left(-2\right)^{n-i}\binom{n+1}{i}\left(\sum\limits_{j=0}^i\binom{i}{j}\binom{2i+j}{n}\right).
\]
\end{thm}
\begin{proof}

Let $D_n\coloneqq \frac{1}{n+1}\sum\limits_{i=0}^n\left(-2\right)^{n-i}\binom{n+1}{i}\left(\sum\limits_{j=0}^i\binom{i}{j}\binom{2i+j}{n}\right)$.

By Lemma \ref{lem:intermediate-id}, it suffices to show that $D_n=\frac{1}{n+1}\sum\limits_{r=1}^{n+1}\binom{2n+2}{r-1}\binom{n+1}{r}2^{r-1}$.

We have
\begin{align*}
\sum\limits_{j=0}^i\binom{i}{j}\binom{2i+j}{n} &=\sum\limits_{j=0}^i\binom{i}{j}[x^n]\left(1+x\right)^{2i+j}\\
&=[x^n]\left(1+x\right)^{2i}\sum\limits_{j=0}^i \binom{i}{j}\left(1+x\right)^j\\
&=[x^n]\left(1+x\right)^{2i}\left(1+(1+x)\right)^i\\
&=[x^n]\left(1+x\right)^{2i}\left(2+x\right)^i
\end{align*}
and hence
\[
D_n=\frac{1}{n+1}[x^n]\sum\limits_{i=0}^n\left(-2\right)^{n-i}\binom{n+1}{i}\left(1+x\right)^{2i}\left(2+x\right)^i.
\]

Write
\begin{equation}\label{eqn:P(x)}
P(x)\coloneqq \sum\limits_{i=0}^n\left(-2\right)^{n-i}\binom{n+1}{i}\left(\left(1+x\right)^2(2+x)\right)^i.
\end{equation}

Then
\begin{equation}\label{eqn:D_n in 3}
D_n=\frac{1}{n+1}[x^n]P(x).
\end{equation}

Multiplying Equation \eqref{eqn:P(x)} by $-2$, we get
\begin{align*}
-2P(x) &=\sum\limits_{i=0}^n\left(-2\right)^{n+1-i}\binom{n+1}{i}\left(\left(1+x\right)^2(2+x)\right)^i\\
& = \left(\sum\limits_{i=0}^{n+1}\left(-2\right)^{n+1-i}\binom{n+1}{i}\left(\left(1+x\right)^2(2+x)\right)^i\right) -\left(\left(1+x\right)^2(2+x)\right)^{n+1}\\
& = \left(\left(1+x\right)^2(2+x)-2\right)^{n+1}-\left(1+x\right)^{2n+2}\left(2+x\right)^{n+1}\\
& = \left(x^3+4x^2+5x\right)^{n+1}-\left(1+x\right)^{2n+2}\left(2+x\right)^{n+1}\\
& = x^{n+1}\left(x^2+4x+5\right)^{n+1}-\left(1+x\right)^{2n+2}\left(2+x\right)^{n+1}.
\end{align*}

Therefore, $P(x)=-\frac{1}{2}\left(x^{n+1}\left(x^2+4x+5\right)^{n+1}-\left(1+x\right)^{2n+2}\left(2+x\right)^{n+1}\right)$. We wish to calculate $[x^n]P(x)$. Notice that the lowest power of $x$ in $x^{n+1}\left(x^2+4x+5\right)^{n+1}$ is $x^{n+1}$.

This implies that the coefficient of $x^n$ in $x^{n+1}\left(x^2+4x+5\right)^{n+1}$ is zero.

Hence
\begin{align*}
[x^n]P(x) &=\frac{1}{2}[x^n]\left(1+x\right)^{2n+2}\left(2+x\right)^{n+1}\\
&=\frac{1}{2}[x^n]\sum\limits_{k=0}^{n+1}\binom{n+1}{k}2^{n+1-k}\sum\limits_{m=0}^{2n+2}\binom{2n+2}{m}x^{m+k} \\
& =\frac{1}{2} \sum\limits_{k=0}^n \binom{n+1}{k}2^{n+1-k}\binom{2n+2}{n-k}\\
& =\frac{1}{2} \sum\limits_{r=1}^{n+1} \binom{n+1}{n+1-r}2^r\binom{2n+2}{r-1} \text{ [writing } r\coloneqq n+1-k]\\
& =\frac{1}{2} \sum\limits_{r=1}^{n+1} \binom{n+1}{r}\binom{2n+2}{r-1}2^r\\
& = \sum\limits_{r=1}^{n+1} \binom{n+1}{r}\binom{2n+2}{r-1}2^{r-1}.
\end{align*}

From Equation \eqref{eqn:D_n in 3}, we get $D_n=\frac{1}{n+1}\sum\limits_{r=1}^{n+1} \binom{2n+2}{r-1}\binom{n+1}{r}2^{r-1}$, thus concluding the proof.
\end{proof}

\section{Evaluation of a diagonal and an application of MacMahon's master theorem}\label{sec 4}

For a reference on diagonals of formal power series, see \cite[2.1 and 2.4]{ACSV}.

\begin{defn}
Let 
\[f\coloneqq \sum\limits_{k_1\ge 0}\cdots\sum\limits_{k_m\ge 0} a_{k_1,\dots,k_m} x_1^{k_1}\cdots x_m^{k_m} \in \mathbb{C}[[x_1,\dots,x_m]].\]
Then the \emph{diagonal} of $f$ is defined to be the power series $\sum\limits_{k\ge 0} a_{k,\dots,k}x^k\in \mathbb{C}[[x]]$.
\end{defn}

\begin{defn}
Let 
\[f(x_1,\dots, x_m)\coloneqq \frac{P(x_1,\dots,x_m)}{Q(x_1,\dots,x_m)}\in \mathbb{C}(x_1,\dots,x_m)\]
be such that $Q(0,\dots,0)\ne 0$. Then $f$ admits a Taylor series expansion 
\[\sum\limits_{k_1\ge 0}\cdots\sum\limits_{k_m\ge 0} c_{k_1,\dots,k_m} x_1^{k_1}\cdots x_m^{k_m}\]
around $(0,\dots,0)$ \emph{(see \cite[2.1]{ACSV})}. The \emph{diagonal} of $f$ is defined to be the diagonal of the Taylor series expansion of $f$ around $(0,\dots,0)$. We denote the coefficient of $x_1^{\alpha_1}\cdots x_m^{\alpha_m}$ in the Taylor series expansion of $f$ around $(0,\dots,0)$ by $[x_1^{\alpha_1}\cdots x_m^{\alpha_m}]f$.
\end{defn}



In \cite{A002897}, Bala conjectured the following result, which we prove.

\begin{thm}\label{thm: GF3}
For $n\ge 0$, we have
\begin{align*}
\binom{2n}{n}^3 &= [\left(xyzt\right)^n]\frac{1}{(1-(x+y)(1-4zt)-z-t)}\\
&= [\left(xyzt\right)^n]\frac{1}{\det(I-\mathrm{diag}(x,y,z,t)M)},    
\end{align*}
where $\mathrm{diag}(x,y,z,t)=\begin{bmatrix}
   x & 0 & 0 & 0\\
   0 & y & 0 & 0\\
   0 & 0 & z & 0\\
   0 & 0 & 0 & t
\end{bmatrix}$ and
$M=\begin{bmatrix}
   1 & 1 & 1 & 1\\
   1 & 1 & 1 & 1\\
   1 & 1 & 1 & -1\\
   1 & 1 & -1 & 1
\end{bmatrix}$.
\end{thm}

\begin{proof}
Let
\[
C_n \coloneqq [\left(xyzt\right)^n]\frac{1}{(1-(x+y)(1-4zt)-z-t)}.
\]

We have that the Taylor series expansion of 
\[\frac{1}{(1-(x+y)(1-4zt)-z-t)} \in \mathbb{C}(x,y,z,t)\]
around $(0,0,0,0)$ is
\begin{align*}
&\sum\limits_{k\ge 0} \left(x+y+z+t-4xzt-4yzt\right)^k\\
&=\sum\limits_{k\ge 0}\sum\limits_{\substack{p_1,\dots, p_6 \ge 0 \\ p_1+\cdots+p_6=k}} \frac{k!}{p_1!\cdots p_6!}x^{p_1}y^{p_2}z^{p_3}t^{p_4}\left(-4xzt\right)^{p_5}\left(-4yzt\right)^{p_6}\\
&=\sum\limits_{k\ge 0}\sum\limits_{\substack{p_1,\dots, p_6 \ge 0 \\ p_1+\cdots+p_6=k}} \frac{k!}{p_1!\cdots p_6!}\left(-4\right)^{p_5+p_6}x^{p_1+p_5}y^{p_2+p_6}z^{p_3+p_5+p_6}t^{p_4+p_5+p_6}.
\end{align*}

Hence
\begin{align*}
C_n &= [\left(xyzt\right)^n] \sum\limits_{k\ge 0}\sum\limits_{\substack{p_1,\dots, p_6 \ge 0 \\ p_1+\cdots+p_6=k}} \frac{k!}{p_1!\cdots p_6!}\left(-4\right)^{p_5+p_6}x^{p_1+p_5}y^{p_2+p_6}z^{p_3+p_5+p_6}t^{p_4+p_5+p_6}\\
&= \sum\limits_{k\ge 0}\sum\limits_{\substack{p_1,\dots, p_6 \ge 0 \\ p_1+\cdots+p_6=k\\ p_1+p_5 = n\\
p_2+p_6 = n\\
p_3+p_5+p_6 = n\\
p_4+p_5+p_6 = n}} \frac{k!}{p_1!\cdots p_6!}\left(-4\right)^{p_5+p_6}.
\end{align*}


From the equations satisfied by $(p_1,\dots,p_6)$, it follows that $p_3=p_4=p_1+p_2-n$, $p_5=n-p_1$, and $p_6=n-p_2$.

Using the fact that $p_1+\cdots+p_6=k$, we get $p_1+p_2+2(p_1+p_2-n)+2n-(p_1+p_2)=k$, which gives $p_1+p_2=\frac{k}{2}$. Thus each $p_i$ can be expressed in terms of $p_1$ and $k$, and additionally, $k$ is forced to be even since $\frac{k}{2}=p_1+p_2$ is an integer. If we let $k=2m$, then the outer sum runs over $m\ge 0$. We can write $p_2=m-p_1$.

The tuple $(p_1,\dots,p_6)$ is thus equal to $(p_1,m-p_1,m-n,m-n,n-p_1,n-m+p_1)$. 

We therefore obtain that
\[
C_n = \sum\limits_{m\ge 0}\sum\limits_{\substack{p_1 \ge 0 \\ m-p_1\ge 0 \\ m-n\ge 0\\ n-p_1 \ge 0\\ n-m+p_1\ge 0}} \frac{(2m)!}{p_1!(m-p_1)!(m-n)!(m-n)!(n-p_1)!(n-m+p_1)!}\left(-4\right)^{2n-m} .
\]

We henceforth write $p$ instead of $p_1$.

The constraints on the inner sum give us that $0\le m-n\le p\le n$. This implies that the outer sum runs from $m=n$ to $m=2n$, and the inner sum runs from $p=m-n$ to $p=n$.

On rewriting, we obtain
\[
C_n = \sum\limits_{m=n}^{2n}\sum\limits_{p=m-n}^n \frac{(2m)!}{p!(m-p)!(m-n)!(m-n)!(n-p)!(n-m+p)!}\left(-4\right)^{2n-m}.
\]

Multiplying the numerator and denominator by $m!(2n-m)!$, we get
\begin{align*}
C_n &= \sum\limits_{m=n}^{2n} \frac{(2m)!}{(m-n)!(m-n)!m!(2n-m)!}\left(-4\right)^{2n-m} \sum\limits_{p=m-n}^n \frac{m!(2n-m)!}{p!(m-p)!(n-p)!(n-m+p)!}\\
&= \sum\limits_{m=n}^{2n} \frac{(2m)!}{(m-n)!(m-n)!m!(2n-m)!}\left(-4\right)^{2n-m} \sum\limits_{p=m-n}^n \binom{m}{p}\binom{2n-m}{n-m+p}.
\end{align*}

We now write $r\coloneqq m-n$, giving
\[
C_n = \sum\limits_{r=0}^n \frac{(2n+2r)!}{r!r!(n+r)!(n-r)!}\left(-4\right)^{n-r} \sum\limits_{p=r}^n \binom{n+r}{p}\binom{n-r}{p-r}.
\]

Writing $s\coloneqq p-r$, we get
\[
C_n = \sum\limits_{r=0}^n \frac{(2n+2r)!}{r!r!(n+r)!(n-r)!}\left(-4\right)^{n-r} \sum\limits_{s=0}^{n-r} \binom{n+r}{s+r}\binom{n-r}{s}.
\]

By \cite[(3.20)]{Gould}, we have 
\[\sum\limits_{s=0}^{n-r} \binom{n+r}{s+r}\binom{n-r}{s}=\binom{2n}{n},\]
which yields
\[
C_n = \binom{2n}{n}\sum\limits_{r=0}^n \frac{(2n+2r)!}{r!r!(n+r)!(n-r)!}\left(-4\right)^{n-r}.
\]

Write 
\[S_n\coloneqq \frac{C_n}{\binom{2n}{n}} = \sum\limits_{r=0}^n \frac{(2n+2r)!}{r!r!(n+r)!(n-r)!}\left(-4\right)^{n-r}.\]

Multiplying the numerator and denominator of $S_n$ by $(n+r)!n!$, we get
\[
S_n = \sum\limits_{r=0}^n \frac{(2n+2r)!}{(n+r)!^2}\frac{(n+r)!}{n!r!}\frac{n!}{r!(n-r)!}\left(-4\right)^{n-r} =
\sum\limits_{r=0}^n \binom{2n+2r}{n+r}\binom{n+r}{n}\binom{n}{r}\left(-4\right)^{n-r}.
\]

Using the identity \cite[(1.10)]{ECombi}, we rewrite the above as
\[
S_n = \sum\limits_{r=0}^n \left(-4\right)^{n+r}\binom{-1/2}{n+r}\binom{n+r}{n}\binom{n}{r}\left(-4\right)^{n-r} = 4^{2n} \sum\limits_{r=0}^n \binom{-1/2}{n+r}\binom{n+r}{n}\binom{n}{r}.
\]
By \cite[(5.21)]{ConcMath}, we have 
\[
\binom{-1/2}{n+r}\binom{n+r}{n}=\binom{-1/2}{n}\binom{-n-\frac{1}{2}}{r}.
\]

Using this gives us
\begin{align*}
S_n &= 4^{2n} \binom{-1/2}{n} \sum\limits_{r=0}^n \binom{-n-\frac{1}{2}}{r}\binom{n}{r}\\
&= 4^{2n} \binom{-1/2}{n} \sum\limits_{r=0}^n \binom{-n-\frac{1}{2}}{r}\binom{n}{n-r}\\
&= 4^{2n} \binom{-1/2}{n}\binom{-1/2}{n},
\end{align*}
where the last equality follows from Vandermonde convolution \cite[(3.1)]{Gould}.

Again using \cite[(1.10)]{ECombi}, we have that 
\[4^{2n} \binom{-1/2}{n}\binom{-1/2}{n}=\binom{2n}{n}^2.\]

We therefore obtain $S_n = \binom{2n}{n}^2$ and hence $C_n = \binom{2n}{n}S_n = \binom{2n}{n}^3$, thereby proving the conjecture.
\end{proof}

The following result will be used to prove a corollary of Theorem \ref{thm: GF3}:

\begin{thm}[MacMahon's master theorem]\cite{MMTshort}\label{thm: MMT}
Let $A=(a_{ij})_{1\le i,j\le m}$ be an $m\times m$ matrix over $\mathbb{C}$. Let $k_1,\dots,k_m\in \mathbb{Z}_{\ge 0}$. Define $L_i(x_1,\dots,x_m)\coloneqq \sum\limits_{j=1}^m a_{ij}x_j$ for $1\le i\le m$. Let $T\coloneqq \mathrm{diag}(x_1,\dots,x_m)$. Then
\[
[x_1^{k_1}\cdots x_m^{k_m}] \prod\limits_{i=1}^m L_i(x_1,\dots,x_m)^{k_i} = [x_1^{k_1}\cdots x_m^{k_m}] \frac{1}{\det(I-TA)}.
\]
\end{thm}

The following result was also conjectured by Bala in \cite{A002897}. According to a comment by Stephan in \cite{A002897}, it was subsequently proved by an autonomous AI agent, with a Lean formalisation of this proof being linked to in \cite{A002897}. Here we present a distinct and independently obtained proof, deriving the result as a direct corollary of Theorem \ref{thm: GF3}.

\begin{cor}
For $n\ge 0$, we have
\[\binom{2n}{n}^3=[\left(xyz\right)^n]\left(1+x+y+z\right)^{2n}\left(1+x+y-z\right)^n\left(1+x-y+z\right)^n.\]
\end{cor}
\begin{proof}
Let
\[
D_n \coloneqq [(xyz)^n]\left(1+x+y+z\right)^{2n}\left(1+x+y-z\right)^n\left(1+x-y+z\right)^n.
\]

By homogenisation, we have
\[
D_n =[\left(xyzt\right)^n]\left(t+x+y+z\right)^{2n}\left(t+x+y-z\right)^n\left(t+x-y+z\right)^n.
\]

By Theorem \ref{thm: MMT}, we therefore have
\[
D_n = [\left(xyzt\right)^n]\frac{1}{\det(I-\mathrm{diag}(x,y,z,t)M)},
\]
and hence, using Theorem \ref{thm: GF3}, the proof is complete.
\end{proof}

\section{An application of Jacobi's residue formula} \label{sec 5}

In \cite{A002897}, Bala conjectured that 
\[
\binom{2n}{n}^3=[\left(xyz\right)^{2n}]\left(1+xy+xz-yz\right)^{2n}\left(1+xy-xz+yz\right)^{2n}\left(1-xy+xz+yz\right)^{2n}.
\]
We prove this conjecture.

\begin{rem}
Let $P_0\coloneqq \{f\in  \mathbb{C}[[x_1,\dots,x_m]]:f\text{ has constant term zero}\}$. We define $\exp: P_0\to \mathbb{C}[[x_1,\dots,x_m]]$ by $f\mapsto \sum\limits_{n\ge 0}\frac{1}{n!}f^n$. For a reference on exponentiation in formal power series rings, see \cite{FormalExp}.
\end{rem}

Some results that will help us:

\begin{lem}\label{lem: c-duality}
Let $A=\left(a_{ij}\right)_{1\le i,j\le m}$ be a symmetric $m\times m$ matrix over $\mathbb{C}$. Let $\alpha_1,\dots,\alpha_m$ and $\beta_1,\dots,\beta_m$ be non-negative integers. Define $L_i(x_1,\dots,x_m)\coloneqq \sum\limits_{j=1}^m a_{ij}x_j$ for $1\le i\le m$. Then,
\[
\frac{[x_1^{\alpha_1}\cdots x_m^{\alpha_m}]\prod\limits_{i=1}^m L_i(x_1,\dots,x_m)^{\beta_i}}{\beta_1!\cdots \beta_m!} = \frac{[x_1^{\beta_1}\cdots x_m^{\beta_m}]\prod\limits_{i=1}^m L_i(x_1,\dots,x_m)^{\alpha_i}}{\alpha_1!\cdots \alpha_m!}.
\]
\end{lem}
\begin{proof}
Consider the formal power series ring $\mathbb{C}[[x_1,\dots,x_m,y_1,\dots,y_m]]$.

Write $x\coloneqq (x_1,\dots, x_m)$ and $y\coloneqq (y_1,\dots, y_m)$.

For $\gamma=(\gamma_1,\dots,\gamma_m)\in \mathbb{Z}^m$, say that $\gamma\ge 0$ if $\gamma_i\ge 0$ for each $1\le i\le m$, and write $x^\gamma$ for the monomial $x_1^{\gamma_1}\cdots x_m^{\gamma_m}$. 

We have
\[
\sum\limits_{i=1}^m \sum\limits_{j=1}^m x_i a_{ij}y_j = \sum\limits_{i=1}^m x_i \sum\limits_{j=1}^m a_{ij}y_j = \sum\limits_{i=1}^m x_i L_i(y),
\]

and therefore

\[
\exp(\sum\limits_{i=1}^m \sum\limits_{j=1}^m x_i a_{ij}y_j) = \exp(\sum\limits_{i=1}^m x_i L_i(y)) = \prod\limits_{i=1}^m \exp(x_iL_i(y)) = \prod\limits_{i=1}^m \sum\limits_{\gamma_i \ge 0} \frac{x_i^{\gamma_i}L_i(y)^{\gamma_i}}{\gamma_i!}.
\]

Writing $\gamma\coloneqq (\gamma_1,\dots,\gamma_m)$, the above expression can be rewritten as
\[
\sum\limits_{\gamma\ge 0} \frac{x_1^{\gamma_1}\cdots x_m^{\gamma_m}}{\gamma_1!\cdots \gamma_m!} \prod\limits_{i=1}^m L_i(y)^{\gamma_i} = \sum\limits_{\gamma\ge 0} \frac{x^\gamma}{\gamma_1!\cdots \gamma_m!} \prod\limits_{i=1}^m L_i(y)^{\gamma_i}.
\]

Let $\alpha\coloneqq (\alpha_1,\dots,\alpha_m)\ge 0$ and $\beta\coloneqq (\beta_1,\dots,\beta_m)\ge 0$.

We have
\begin{align*}
[x^\alpha y^\beta] \exp(\sum\limits_{i=1}^m \sum\limits_{j=1}^m x_i a_{ij}y_j) &= [x^\alpha y^\beta] \sum\limits_{\gamma\ge 0} \frac{x^\gamma}{\gamma_1!\cdots \gamma_m!} \prod\limits_{i=1}^m L_i(y)^{\gamma_i}\\
&= [x^\alpha y^\beta] \frac{x^\alpha}{\alpha_1!\cdots \alpha_m!} \prod\limits_{i=1}^m L_i(y)^{\alpha_i}\\
&= [y^\beta] \frac{\prod\limits_{i=1}^m L_i(y)^{\alpha_i}}{\alpha_1!\cdots \alpha_m!}.
\end{align*}

Since $A$ is symmetric, we have $a_{ij}=a_{ji}$, and hence
\[
\sum\limits_{i=1}^m \sum\limits_{j=1}^m x_i a_{ij}y_j =
\sum\limits_{i=1}^m \sum\limits_{j=1}^m x_i a_{ji}y_j =
\sum\limits_{j=1}^m \sum\limits_{i=1}^m x_i a_{ji}y_j =
\sum\limits_{j=1}^m y_j \sum\limits_{i=1}^m a_{ji}x_i =
\sum\limits_{j=1}^m y_j L_j(x).
\]

We have
\begin{align*}
[y^\beta] \frac{\prod\limits_{i=1}^m L_i(y)^{\alpha_i}}{\alpha_1!\cdots \alpha_m!} &= [x^\alpha y^\beta] \exp(\sum\limits_{i=1}^m \sum\limits_{j=1}^m x_i a_{ij}y_j)\\
&= [x^\alpha y^\beta] \exp(\sum\limits_{j=1}^m y_j L_j(x))\\
&=[x^\alpha y^\beta] \prod\limits_{j=1}^m \exp(y_jL_j(x))\\
&=[x^\alpha y^\beta] \prod\limits_{j=1}^m \sum\limits_{\gamma_j \ge 0} \frac{y_j^{\gamma_j}L_j(x)^{\gamma_j}}{\gamma_j!}\\
&= [x^\alpha y^\beta] \sum\limits_{\gamma\ge 0} \frac{y_1^{\gamma_1}\cdots y_m^{\gamma_m}}{\gamma_1!\cdots \gamma_m!} \prod\limits_{j=1}^m L_j(x)^{\gamma_j}\\
&=[x^\alpha y^\beta] \sum\limits_{\gamma\ge 0} \frac{y^\gamma}{\gamma_1!\cdots \gamma_m!} \prod\limits_{j=1}^m L_j(x)^{\gamma_j}\\
&= [x^\alpha y^\beta] \frac{y^\beta}{\beta_1!\cdots \beta_m!} \prod\limits_{j=1}^m L_j(x)^{\beta_j}\\
&= [x^\alpha] \frac{\prod\limits_{j=1}^m L_j(x)^{\beta_j}}{\beta_1!\cdots \beta_m!}. \qedhere
\end{align*}
\end{proof}

\begin{thm}[Jacobi's residue formula]\cite[Theorem~3]{GesselResidue}\label{thm: Jacobi residue}
Let $f_i(x_1,\dots,x_m)$ $(1\le i\le m)$ be Laurent series and $p_{ij}$ $(1\le i,j\le m)$ be integers such that
\[\frac{f_i(x_1,\dots,x_m)}{x_1^{p_{i1}}\cdots x_m^{p_{im}}}\]
are formal power series with nonzero constant term. Suppose that $\Phi(t_1,\dots,t_m)$ is a Laurent series. Let
\[\frac{\partial(f_1,\dots, f_m)}{\partial(x_1,\dots,x_m)}\coloneqq \det\left(\frac{\partial f_i}{\partial x_j}\right)_{1\le i,j\le m} \text{ (the Jacobian determinant) }\]
and $A\coloneqq (p_{ij})_{1\le i,j\le m}$. Then
\[
[x_1^{-1}\cdots x_m^{-1}] \frac{\partial(f_1,\dots, f_m)}{\partial(x_1,\dots,x_m)} \Phi(f_1,\dots,f_m) = \det(A) [t_1^{-1}\cdots t_m^{-1}] \Phi(t_1,\dots,t_m).
\]
\end{thm}

We are now ready to prove the conjecture.

\begin{thm}
For $n\ge 0$, we have
\[\binom{2n}{n}^3 =[\left(xyz\right)^{2n}]\left(1+xy+xz-yz\right)^{2n}\left(1+xy-xz+yz\right)^{2n}\left(1-xy+xz+yz\right)^{2n}. \]
\end{thm}
\begin{proof}
Let
\[
C_n \coloneqq [\left(xyz\right)^{2n}]\left(1+xy+xz-yz\right)^{2n}\left(1+xy-xz+yz\right)^{2n}\left(1-xy+xz+yz\right)^{2n}.
\]

Suppose that $\left(xyz\right)^{2n}=\left(xy\right)^{r_1}\left(xz\right)^{r_2}\left(yz\right)^{r_3}$ for integers $r_1,r_2,r_3$. By comparing degrees, we have $r_1+r_2=r_1+r_3=r_2+r_3=2n$. Routine calculations then give us $r_1=r_2=r_3=n$, which implies that an $\left(xyz\right)^{2n}$ term can only be obtained by multiplying $\left(xy\right)^n$, $\left(xz\right)^n$, and $\left(yz\right)^n$ terms. Hence $C_n$ is also the coefficient of $\left(x_1y_1z_1\right)^{n}$ in 
\[\left(1+x_1+y_1-z_1\right)^{2n}\left(1+x_1-y_1+z_1\right)^{2n}\left(1-x_1+y_1+z_1\right)^{2n}.
\]

We relabel $x_1,y_1,$ and $z_1$ as $x,y,$ and $z$ respectively for convenience.

We therefore have that
\begin{align*}
C_n &=[(xyz)^n]\left(1+x+y-z\right)^{2n}\left(1+x-y+z\right)^{2n}\left(1-x+y+z\right)^{2n}\\
&= [t^{3n}x^ny^nz^n]\left(t+x+y-z\right)^{2n}\left(t+x-y+z\right)^{2n}\left(t-x+y+z\right)^{2n}\\
&= [t^{3n}x^ny^nz^n]\left(t+x+y+z\right)^0\left(t+x+y-z\right)^{2n}\left(t+x-y+z\right)^{2n}\left(t-x+y+z\right)^{2n},
\end{align*}

where the second-last equality follows from homogenisation.

By Lemma \ref{lem: c-duality}, we have that

\begin{align*}
\frac{C_n}{(2n)!^3} &=\frac{[t^{3n}x^ny^nz^n]\left(t+x+y+z\right)^0\left(t+x+y-z\right)^{2n}\left(t+x-y+z\right)^{2n}\left(t-x+y+z\right)^{2n}}{(2n)!^3}\\
&=\frac{[t^0x^{2n}y^{2n}z^{2n}]\left(t+x+y+z\right)^{3n}\left(t+x+y-z\right)^n\left(t+x-y+z\right)^n\left(t-x+y+z\right)^n}{(3n)!(n!)^3},
\end{align*}

and therefore that
\begin{align*}
C_n &=\frac{(2n)!^3[t^0x^{2n}y^{2n}z^{2n}]\left(t+x+y+z\right)^{3n}\left(t+x+y-z\right)^n\left(t+x-y+z\right)^n\left(t-x+y+z\right)^n}{(3n)!(n!)^3}\\
&=\frac{(2n)!^3[x^{2n}y^{2n}z^{2n}]\left(x+y+z\right)^{3n}\left(x+y-z\right)^n\left(x-y+z\right)^n\left(-x+y+z\right)^n}{(3n)!(n!)^3}.
\end{align*}

Let 
\begin{equation}\label{eqn:K_n}
K_n\coloneqq \frac{(3n)!(n!)^3 C_n}{(2n)!^3}.
\end{equation}

Then
\begin{align*}
K_n &= [x^{2n}y^{2n}z^{2n}]\left(x+y+z\right)^{3n}\left(x+y-z\right)^n\left(x-y+z\right)^n\left(-x+y+z\right)^n\\
&= [x^{2n}y^{2n}]\left(x+y+1\right)^{3n}\left(x+y-1\right)^n\left(x-y+1\right)^n\left(-x+y+1\right)^n.
\end{align*}

We evaluate $K_n$ by means of a substitution of variables and Theorem \ref{thm: Jacobi residue}.

Let
\[
x=x(U,V)=\frac{U(1-V)}{2-U-V} \text{, } y=y(U,V)=\frac{V(1-U)}{2-U-V},
\]

and
\[
F(x,y)\coloneqq \left(x+y+1\right)^{3n}\left(x+y-1\right)^n\left(x-y+1\right)^n\left(-x+y+1\right)^n.
\]

Then
\begin{equation}\label{eqn:K_n-Laurent}
K_n=[x^{2n}y^{2n}]F(x,y)=[x^{-1}y^{-1}] \frac{F(x,y)}{x^{2n+1}y^{2n+1}}.
\end{equation}

We have:
\begin{align*}
x+y+1 &= \frac{U(1-V)+V(1-U)+2-U-V}{2-U-V} = \frac{2(1-UV)}{2-U-V};\\
x+y-1 &= \frac{U+V-2UV+U+V-2}{2-U-V} = -\frac{2(1-U)(1-V)}{2-U-V};\\
x-y+1 &= \frac{U-UV-V+UV+2-U-V}{2-U-V} = \frac{2(1-V)}{2-U-V};\\
-x+y+1 &= \frac{-U+UV+V-UV+2-U-V}{2-U-V} = \frac{2(1-U)}{2-U-V}.
\end{align*}

This gives
\begin{align*}
F(x(U,V),y(U,V)) &= \left(x+y+1\right)^{3n}\left(x+y-1\right)^n\left(x-y+1\right)^n\left(-x+y+1\right)^n\\
&= \frac{2^{3n}\left(1-UV\right)^{3n}\left(-2\right)^n\left(1-U\right)^n\left(1-V\right)^n2^n\left(1-V\right)^n2^n\left(1-U\right)^n}{\left(2-U-V\right)^{6n}}\\
&= \frac{\left(-1\right)^n2^{6n}\left(1-UV\right)^{3n}\left(1-U\right)^{2n}\left(1-V\right)^{2n}}{\left(2-U-V\right)^{6n}}.
\end{align*}

We also have

\[x^{2n+1}y^{2n+1}=\frac{U^{2n+1}V^{2n+1}\left(1-U\right)^{2n+1}\left(1-V\right)^{2n+1}}{\left(2-U-V\right)^{4n+2}}.\]

Hence
\begin{equation}
\label{eqn:K_n-Laurent-expanded}
\begin{aligned}
\frac{F(x(U,V),y(U,V))}{x\left(U,V\right)^{2n+1}y\left(U,V\right)^{2n+1}}
&= \frac{\left(-1\right)^n2^{6n}\left(1-UV\right)^{3n}\left(1-U\right)^{2n}\left(1-V\right)^{2n}\left(2-U-V\right)^{4n+2}}{\left(2-U-V\right)^{6n}U^{2n+1}V^{2n+1}\left(1-U\right)^{2n+1}\left(1-V\right)^{2n+1}}\\
&=\frac{\left(-1\right)^n2^{6n}\left(1-UV\right)^{3n}}{\left(2-U-V\right)^{2n-2}U^{2n+1}V^{2n+1}\left(1-U\right)\left(1-V\right)}.
\end{aligned}
\end{equation}

By Theorem \ref{thm: Jacobi residue}, we have
\begin{equation}\label{eqn:applying-jacobi}
\det(A)\;[x^{-1}y^{-1}] \frac{F(x,y)}{x^{2n+1}y^{2n+1}} = [U^{-1}V^{-1}] \frac{\partial(x,y)}{\partial(U,V)} \frac{F(x(U,V),y(U,V))}{x\left(U,V\right)^{2n+1}y\left(U,V\right)^{2n+1}},
\end{equation}

where $A=\begin{pmatrix}
    p_{11} & p_{12}\\
    p_{21} & p_{22}
\end{pmatrix}$, for integers $p_{ij}$ such that
\[
\frac{x(U,V)}{U^{p_{11}}V^{p_{12}}} \text{ and } \frac{y(U,V)}{U^{p_{21}}V^{p_{22}}}
\]
are formal power series with nonzero constant term.

We have 
\[\frac{\partial(x,y)}{\partial(U,V)} = \det\begin{pmatrix}
    \frac{\partial x}{\partial U} & \frac{\partial x}{\partial V}\\
    \frac{\partial y}{\partial U} & \frac{\partial y}{\partial V}
\end{pmatrix} = \frac{\partial x}{\partial U}\frac{\partial y}{\partial V}-\frac{\partial x}{\partial V}\frac{\partial y}{\partial U}.\]

We evaluate the required partial derivatives.
\begin{align*}
\frac{\partial x}{\partial U} &=\frac{\partial}{\partial U}\left(\frac{U(1-V)}{2-U-V}\right) =\frac{(2-V)(1-V)}{\left(2-U-V\right)^2}\\
\frac{\partial y}{\partial V} &=\frac{\partial}{\partial V}\left(\frac{V(1-U)}{2-U-V}\right) =\frac{(2-U)(1-U)}{\left(2-U-V\right)^2}\\
\frac{\partial x}{\partial V} &=\frac{\partial}{\partial V}\left(\frac{U(1-V)}{2-U-V}\right) =\frac{U(U-1)}{\left(2-U-V\right)^2}\\
\frac{\partial y}{\partial U} &=\frac{\partial}{\partial U}\left(\frac{V(1-U)}{2-U-V}\right) =\frac{V(V-1)}{\left(2-U-V\right)^2}.
\end{align*}

We obtain
\begin{equation}\label{eqn:Jac-det}
\frac{\partial(x,y)}{\partial(U,V)} = \frac{(2-V)(1-V)(2-U)(1-U)-UV(1-U)(1-V)}{\left(2-U-V\right)^4} = \frac{2(1-U)(1-V)}{\left(2-U-V\right)^3}.
\end{equation}

We now determine the entries of $A$. We have
\begin{align*}
x(U,V) &= \frac{U(1-V)}{2-U-V} = \frac{U(1-V)}{2}\frac{1}{1-\left(\frac{U+V}{2}\right)}\\
&= \frac{U(1-V)}{2}\left(1+\left(\frac{U+V}{2}\right)+\left(\frac{U+V}{2}\right)^2+\cdots\right),
\end{align*}

and therefore, in order for 
\[\frac{x(U,V)}{U^{p_{11}}V^{p_{12}}}\]
to be a formal power series with nonzero constant term, we must have $p_{11}=1$ and $p_{12}=0$.

Similarly,
\begin{align*}
y(U,V) &= \frac{V(1-U)}{2-U-V} = \frac{V(1-U)}{2}\frac{1}{1-\left(\frac{U+V}{2}\right)}\\
&= \frac{V(1-U)}{2}\left(1+\left(\frac{U+V}{2}\right)+\left(\frac{U+V}{2}\right)^2+\cdots\right).
\end{align*}

In order for 
\[\frac{y(U,V)}{U^{p_{21}}V^{p_{22}}}\]
to be a formal power series with nonzero constant term, we must have $p_{21}=0$ and $p_{12}=1$.

Therefore, $A=\begin{pmatrix}
    1 & 0\\
    0 & 1
\end{pmatrix}$, and hence $\det(A)=1$.

From equations \eqref{eqn:K_n-Laurent} and \eqref{eqn:applying-jacobi}, we obtain
\begin{align*}
K_n &=[U^{-1}V^{-1}] \frac{\partial(x,y)}{\partial(U,V)} \frac{F(x(U,V),y(U,V))}{x(U,V)^{2n+1}y(U,V)^{2n+1}}\\
&= [U^{-1}V^{-1}]\frac{2(1-U)(1-V)\left(2-U-V\right)^{4n+2}\left(-1\right)^n2^{6n}\left(1-UV\right)^{3n}\left(1-U\right)^{2n}\left(1-V\right)^{2n}}{\left(2-U-V\right)^3U^{2n+1}V^{2n+1}\left(1-U\right)^{2n+1}\left(1-V\right)^{2n+1}\left(2-U-V\right)^{6n}}\\
&= [U^{-1}V^{-1}]\frac{\left(-1\right)^n2^{6n+1}\left(1-UV\right)^{3n}}{U^{2n+1}V^{2n+1}\left(2-U-V\right)^{2n+1}} = [U^{2n}V^{2n}]\frac{\left(-1\right)^n2^{6n+1}\left(1-UV\right)^{3n}}{\left(2-U-V\right)^{2n+1}},
\end{align*}
where the second equality follows from equations \eqref{eqn:K_n-Laurent-expanded} and \eqref{eqn:Jac-det}. 

Let
\begin{equation}\label{eqn:R_n}
R_n\coloneqq \left(-1\right)^n K_n = 2^{6n+1}[U^{2n}V^{2n}]\frac{\left(1-UV\right)^{3n}}{\left(2-U-V\right)^{2n+1}}.
\end{equation}

We have $\left(1-UV\right)^{3n}=\sum\limits_{r=0}^{3n}\binom{3n}{r}\left(-1\right)^r\left(UV\right)^r$.

Only $0\le r\le 2n$ contributes to the coefficient of $U^{2n}V^{2n}$, and hence
\begin{equation}\label{eqn:R_n-expanded-1}
R_n = 2^{6n+1}\sum\limits_{r=0}^{2n}\binom{3n}{r}\left(-1\right)^r [U^{2n-r}V^{2n-r}]\left(2-U-V\right)^{-2n-1}.
\end{equation}

We have
\[\left(2-U-V\right)^{-2n-1}=2^{-2n-1}(1-\left(\frac{U+V}{2}\right))^{-2n-1}= 2^{-2n-1}\sum\limits_{k\ge 0}\binom{k+2n}{k}\frac{\left(U+V\right)^k}{2^k},
\]

where the last equality follows from \cite[(2.5.7)]{gfology}.

This gives
\begin{align*}
[U^{2n-r}V^{2n-r}]\left(2-U-V\right)^{-2n-1} &= 2^{-2n-1}[U^{2n-r}V^{2n-r}]\sum\limits_{k\ge 0}\binom{k+2n}{k}\frac{\left(U+V\right)^k}{2^k}\\
&= 2^{-2n-1}\binom{4n-2r+2n}{4n-2r}\frac{\binom{4n-2r}{2n-r}}{2^{4n-2r}} = \frac{\binom{6n-2r}{4n-2r}\binom{4n-2r}{2n-r}}{2^{6n-2r+1}}.
\end{align*}

Equation \eqref{eqn:R_n-expanded-1} therefore yields
\begin{equation}\label{eqn:R_n-expanded-2}
\begin{aligned}
R_n &= 2^{6n+1}\sum\limits_{r=0}^{2n} \left(-1\right)^r\binom{3n}{r}\frac{\binom{6n-2r}{4n-2r}\binom{4n-2r}{2n-r}}{2^{6n-2r+1}}\\
&= \sum\limits_{r=0}^{2n} \left(-1\right)^r4^r\binom{3n}{r}\binom{6n-2r}{4n-2r}\binom{4n-2r}{2n-r}\\
&= \sum\limits_{r=0}^{2n} \left(-1\right)^r4^r\binom{3n}{r} \frac{(6n-2r)!}{(2n)!(4n-2r)!} \frac{(4n-2r)!}{(2n-r)!^2}\\
&= \sum\limits_{r=0}^{2n} \left(-1\right)^r4^r\binom{3n}{r} \frac{(6n-2r)!}{(2n)!(2n-r)!^2}.
\end{aligned}
\end{equation}

Recall the \emph{Pochhammer symbol}
\[
\left(a\right)_r\coloneqq a(a+1)\cdots(a+r-1) \text{ \cite[5.2.4]{Handbook}.}
\]
We will use it to simplify the final summation in Equation \eqref{eqn:R_n-expanded-2}.

We have
\begin{equation}\label{eqn:poch-1}
\begin{aligned}
\left(\frac{1}{2}-3n\right)_r &= \left(\frac{1}{2}-3n\right)\left(\frac{1}{2}-3n+1\right)\cdots \left(\frac{1}{2}-3n+r-1\right)\\
&= \left(\frac{1-6n}{2}\right)\left(\frac{1-6n+2}{2}\right)\cdots \left(\frac{1-6n+2r-2}{2}\right)\\
&= \frac{1}{2^r}(1-6n)(3-6n)\cdots (2r-1-6n)\\
&= \frac{\left(-1\right)^r}{2^r}(6n-1)(6n-3)\cdots (6n-2r+1).
\end{aligned}
\end{equation}

Using Equation \eqref{eqn:poch-1}, we get
\begin{equation}\label{eqn:poch-2}
\begin{aligned}
\left(-1\right)^r4^r\binom{3n}{r} (6n-2r)! &= \left(-1\right)^r4^r \frac{(6n-2r)!(3n)!}{r!(3n-r)!}\\
&= \left(-1\right)^r2^r \frac{(6n-2r)!2^r(3n-r+1)(3n-r+2)\cdots (3n)}{r!}\\
&= \left(-1\right)^r2^r \frac{(6n-2r)!(6n-2r+2)(6n-2r+4)\cdots (6n)}{r!}\\
&= \left(-1\right)^r2^r\frac{(6n-2r)!(6n-2r+1)(6n-2r+2)\cdots (6n)}{r!(6n-2r+1)(6n-2r+3)\cdots (6n-1)}\\
&= \left(-1\right)^r2^r\frac{(6n)!}{r!(6n-2r+1)(6n-2r+3)\cdots (6n-1)}\\
&= \frac{(6n)!}{r!\left(\frac{1}{2}-3n\right)_r}.
\end{aligned}
\end{equation}

We also have
\begin{equation}\label{eqn:poch-3}
\begin{aligned}
\frac{\left(-1\right)^r}{(2n-r)!} &= \frac{\left(-1\right)^r(2n)(2n-1)\cdots (2n-r+1)}{(2n)!}\\
&= \frac{(-2n)(-2n+1)\cdots (-2n+r-1)}{(2n)!}\\
&= \frac{\left(-2n\right)_r}{(2n)!}.
\end{aligned}
\end{equation}

Note that squaring both sides in Equation \eqref{eqn:poch-3} yields
\begin{equation}\label{eqn:poch-4}
\frac{1}{(2n-r)!^2} = \frac{\left(-2n\right)_r^2}{(2n)!^2}.
\end{equation}

We can finally simplify the expression for $R_n$ obtained in Equation \eqref{eqn:R_n-expanded-2} as follows, using equations \eqref{eqn:poch-2} and \eqref{eqn:poch-4}:
\begin{equation}\label{eqn:R_n-expanded-3}
\begin{aligned}
R_n &= \sum\limits_{r=0}^{2n} \left(-1\right)^r4^r\binom{3n}{r}(6n-2r)! \frac{1}{(2n-r)!^2} \frac{1}{(2n)!}\\
&= \sum\limits_{r=0}^{2n} \frac{(6n)!}{r!\left(\frac{1}{2}-3n\right)_r} \frac{(-2n)_r^2}{(2n)!^2} \frac{1}{(2n)!}\\
&= \frac{(6n)!}{(2n)!^3} \sum\limits_{r=0}^{2n} \frac{(-2n)_r^2}{r!\left(\frac{1}{2}-3n\right)_r}.
\end{aligned}
\end{equation}

Using Equation \eqref{eqn:R_n-expanded-3}, let
\begin{equation}\label{eqn:P_n}
P_n \coloneqq \frac{(2n)!^3}{(6n)!}R_n = \sum\limits_{r=0}^{2n} \frac{\left(-2n\right)_r^2}{r!\left(\frac{1}{2}-3n\right)_r}.
\end{equation}

Recall that the \emph{hypergeometric function} $_2F_1(a,b;c;z)\coloneqq \sum\limits_{s\ge 0}\frac{\left(a\right)_s\left(b\right)_s}{\left(c\right)_s}z^s$ (\cite[15.1.1 and 15.2.1]{Handbook}) satisfies
\begin{equation}\label{eqn:hypergeom}
_2F_1(-m,b;c;z) = \sum\limits_{r=0}^m\left(-1\right)^r\binom{m}{r}\frac{\left(b\right)_r}{\left(c\right)_r}z^r \text{ \cite[15.2.4]{Handbook}}
\end{equation}
when $m$ is a non-negative integer and $c$ is not a negative integer.

We therefore have
\begin{equation}\label{eqn:P_n-expanded-1}
\begin{aligned}
P_n &= \sum\limits_{r=0}^{2n} \frac{\left(-2n\right)_r^2}{r!\left(\frac{1}{2}-3n\right)_r} \text{ [From Equation \eqref{eqn:P_n}]}\\
&=\sum\limits_{r=0}^{2n} \frac{\left(-2n\right)_r}{(2n)!} \frac{(2n)!}{r!} \frac{\left(-2n\right)_r}{\left(\frac{1}{2}-3n\right)_r}\\
&=\sum\limits_{r=0}^{2n} \frac{\left(-1\right)^r}{(2n-r)!} \frac{(2n)!}{r!} \frac{\left(-2n\right)_r}{\left(\frac{1}{2}-3n\right)_r} \text{ [From Equation \eqref{eqn:poch-3}]}\\
&= \sum\limits_{r=0}^{2n} \left(-1\right)^r\binom{2n}{r} \frac{\left(-2n\right)_r}{\left(\frac{1}{2}-3n\right)_r}\\
&\text{= } _2F_1(-2n,2n;\frac{1}{2}-3n;1) \text{ [From Equation \eqref{eqn:hypergeom}]}.
\end{aligned}
\end{equation}

The Chu-Vandermonde identity \cite[15.4.24]{Handbook} states that 
\begin{equation}\label{eqn:Chu-Vandermonde}
_2F_1(-m,b;c;1)=\frac{\left(c-b\right)_m}{\left(c\right)_m}
\end{equation}
under the same hypotheses as Equation \eqref{eqn:hypergeom}.

Combining equations \eqref{eqn:hypergeom} and \eqref{eqn:Chu-Vandermonde} yields
\begin{equation}\label{eqn:P_n-expanded-2}
P_n\text{ = }_2F_1(-2n,2n;\frac{1}{2}-3n;1) = \frac{\left(\frac{1}{2}-n\right)_{2n}}{\left(\frac{1}{2}-3n\right)_{2n}}.
\end{equation}

We will now simplify $\left(\frac{1}{2}-n\right)_{2n}$ and $\left(\frac{1}{2}-3n\right)_{2n}$.

The former can be written as
\begin{equation}\label{eqn:poch-num}
\begin{aligned}
\left(\frac{1}{2}-n\right)_{2n} &= \left(\frac{1}{2}-n\right)\left(\frac{1}{2}-n+1\right)\cdots \left(\frac{1}{2}-n+2n-1\right)\\
&= \left(\frac{1-2n}{2}\right)\left(\frac{1-2n+2}{2}\right)\cdots \left(\frac{1-2n+4n-2}{2}\right)\\
&= \frac{(1-2n)(3-2n)\cdots (2n-1)}{2^{2n}}\\
&= \frac{\left(-1\right)^n\left((1)(3)\cdots (2n-1)\right)^2}{2^{2n}}\\
&= \left(-1\right)^n \left(\left(\frac{1}{2}\right)\left(\frac{3}{2}\right)\cdots \left(\frac{2n-1}{2}\right)\right)^2\\
&= \left(-1\right)^n\left(\frac{1}{2}\left(\frac{1}{2}+1\right)\cdots \left(\frac{1}{2}+n-1\right)\right)^2\\
&= \left(-1\right)^n\left(\frac{1}{2}\right)_n^2.
\end{aligned}
\end{equation}

The latter simplifies to
\begin{equation}\label{eqn:poch-den}
\begin{aligned}
\left(\frac{1}{2}-3n\right)_{2n} &= \left(\frac{1-6n}{2}\right)\left(\frac{1-6n+2}{2}\right)\cdots \left(\frac{1-6n+4n-2}{2}\right)\\
&= \frac{(1-6n)(3-6n)\cdots (-2n-1)}{2^{2n}}\\
&= \frac{\left(-1\right)^{2n}(2n+1)(2n+3)\cdots (6n-1)}{2^{2n}}\\
&= \frac{(2n+1)(2n+3)\cdots (6n-1)}{2^{2n}}\\
&= \frac{(1)(3)\cdots (6n-1)}{2^{3n}} \frac{2^n}{(1)(3)\cdots (2n-1)}\\
&= \frac{\left(\frac{1}{2}\right)\left(\frac{3}{2}\right)\cdots \left(\frac{6n-1}{2}\right)}{\left(\frac{1}{2}\right)\left(\frac{3}{2}\right)\cdots \left(\frac{2n-1}{2}\right)}\\
&= \frac{\left(\frac{1}{2}\right)_{3n}}{\left(\frac{1}{2}\right)_{n}}.
\end{aligned}
\end{equation}

Combining equations \eqref{eqn:P_n-expanded-2}, \eqref{eqn:poch-num}, and \eqref{eqn:poch-den} yields
\begin{equation}\label{eqn:P_n-expanded-3}
P_n = \left(-1\right)^n \frac{\left(\frac{1}{2}\right)_{n}^3}{\left(\frac{1}{2}\right)_{3n}}.
\end{equation}

We have
\begin{equation}\label{eqn:poch-5}
\begin{aligned}
\left(\frac{1}{2}\right)_{n} &= \left(\frac{1}{2}\right)\left(\frac{3}{2}\right)\cdots \left(\frac{2n-1}{2}\right)\\
&= (1)(3)\cdots (2n-1) \frac{(2^n) n!}{(4^n) n!}\\
&= \frac{((1)(3)\cdots (2n-1))((2)(4)\cdots (2n))}{(4^n)n!}\\
&= \frac{(2n)!}{(4^n)n!}.
\end{aligned}
\end{equation}

From equations \eqref{eqn:P_n-expanded-3} and \eqref{eqn:poch-5}, we obtain
\begin{equation}\label{eqn:P_n-final}
P_n = \left(-1\right)^n \frac{\left(\frac{1}{2}\right)_{n}^3}{\left(\frac{1}{2}\right)_{3n}} = \left(-1\right)^n \frac{(2n)!^3}{(4^{3n})n!^3} \frac{(4^{3n})(3n)!}{(6n)!} = \left(-1\right)^n \frac{(2n)!^3}{n!^3} \frac{(3n)!}{(6n)!}.
\end{equation}

Using equations \eqref{eqn:P_n} and \eqref{eqn:P_n-final}, we get
\begin{equation}\label{eqn:R_n-final}
R_n = \frac{(6n)!}{(2n)!^3} \left(-1\right)^n \frac{(2n)!^3}{n!^3} \frac{(3n)!}{(6n)!} = \left(-1\right)^n \frac{(3n)!}{n!^3}.
\end{equation}

Equations \eqref{eqn:R_n} and \eqref{eqn:R_n-final} then give
\begin{equation}\label{eqn:K_n-final}
K_n=\left(-1\right)^n R_n = \frac{(3n)!}{n!^3}.
\end{equation}

Finally, equations \eqref{eqn:K_n} and \eqref{eqn:K_n-final} yield
\begin{equation}\label{eqn:C_n-final}
C_n = \frac{(2n)!^3}{(3n)!(n!)^3} K_n = \frac{(2n)!^3}{(3n)!(n!)^3} \frac{(3n)!}{n!^3} = \frac{(2n)!^3}{(n!)^6}=\binom{2n}{n}^3,
\end{equation}

thus proving the theorem.
\end{proof}

\section{On the square root of a generating function}\label{sec 6}

\begin{lem}\label{lem:sqrt-gf}
Let $A(x)=\sum\limits_{n\ge 0}a_nx^n \in \mathbb{C}[[x]]$ such that $a_0\ne 0$. Then there exists $B(x)=\sum\limits_{n\ge 0}b_nx^n \in \mathbb{C}[[x]]$ such that $B(x)^2=A(x)$.
\end{lem}

\begin{proof}
We construct the required $B(x)$ by defining $b_n$ inductively.

For $A(x),B(x)\in\mathbb{C}[[x]]]$, we have, by \cite[(2.2.3)]{gfology}, that $B(x)^2=A(x)$ if and only if $a_n=\sum\limits_{k=0}^nb_kb_{n-k}$ for any $n\ge 0$.

Define $b_0$ to be the principal square root of $a_0$. Note that $b_0\ne 0$. Then $a_0=b_0^2=\sum\limits_{k=0}^0b_kb_{-k}$.

Define $b_n\coloneqq \frac{1}{2b_0}\left(a_n-\sum\limits_{k=1}^{n-1}b_kb_{n-k}\right)$ for $n\ge 1$. Then $a_n = 2b_0b_n + \sum\limits_{k=1}^{n-1}b_kb_{n-k} =\sum\limits_{k=0}^nb_kb_{n-k}$ for $n\ge 1$.

Hence $a_n =\sum\limits_{k=0}^nb_kb_{n-k}$ for $n\ge 0$. We have thus constructed the required square root $B(x)$ of $A(x)$.
\end{proof}


Let $a_n\coloneqq \binom{n}{k}^2\binom{2k}{n}^2$. In \cite{A290575}, Chapoton conjectured the following result, which we prove.

\begin{thm}
The ordinary power series generating function $A(x)$ of $\{a_n\}_{n\ge 0}$ has a square root with integer coefficients.
\end{thm}
\begin{proof}
We have 
\[\binom{n}{k}\binom{2k}{n}=\frac{n!}{k!(n-k)!}\frac{(2k)!}{n!(2k-n)!}=\frac{(2k)!}{k!k!}\frac{k!}{(n-k)!(2k-n)!}=\binom{2k}{k}\binom{k}{n-k}.\]

This gives $a_0=\sum\limits_{k=0}^0\binom{2k}{k}^2\binom{k}{-k}^2=\binom{0}{0}^4=1$.

Let $B(x)=\sum\limits_{n\ge 0}b_nx^n$ be the square root of $A(x)$ constructed in Lemma \ref{lem:sqrt-gf}. Then $b_0=1$ and $b_n= \frac{1}{2b_0}\left(a_n-\sum\limits_{k=1}^{n-1}b_kb_{n-k}\right)$ for $n\ge 1$.

We have $a_n=\sum\limits_{k=0}^nb_kb_{n-k}$ for $n\ge 0$. Suppose $n\ge 1$. Then $a_n=\sum\limits_{k=0}^n\binom{2k}{k}^2\binom{k}{n-k}^2$.

For $k=0$, $\binom{k}{n-k}=0$ and hence $\binom{2k}{k}^2\binom{k}{n-k}^2\equiv 0\;(\mathrm{mod}\;4)$. For $k\ge 1$, we have 
\[\binom{2k}{k}=\frac{(2k)!}{k!k!}=2 \frac{(2k-1)!}{(k-1)!k!}=2\binom{2k-1}{k-1},\]
and hence $\binom{2k}{k}^2\binom{k}{n-k}^2\equiv 0\;(\mathrm{mod}\;4)$.

This implies that $a_n=\sum\limits_{k=0}^n\binom{2k}{k}^2\binom{k}{n-k}^2 \equiv 0\;(\mathrm{mod}\;4)$ for $n\ge 1$.

We now prove by induction that $b_n$ is an even integer for $n\ge 1$.

We have $a_1=\sum\limits_{k=0}^1\binom{2k}{k}^2\binom{k}{1-k}^2=\binom{0}{0}^2\binom{0}{1}^2+\binom{2}{1}^2\binom{1}{0}^2=4$.

As defined above, $b_n=\frac{1}{2b_0}\left(a_n-\sum\limits_{k=1}^{n-1}b_kb_{n-k}\right)=\frac{1}{2}\left(a_n-\sum\limits_{k=1}^{n-1}b_kb_{n-k}\right)$.

Therefore, $b_1=\frac{a_1}{2}=2$, thus proving the base case.

For the inductive step, fix some $n\ge 2$ and suppose that $b_k$ is even for $1\le k\le n-1$.

Then $b_kb_{n-k}\equiv 0\;(\mathrm{mod}\;4)$ for each $1\le k\le n-1$.

We have already proved above that $a_n\equiv 0\;(\mathrm{mod}\;4)$ for $n\ge 1$. This implies that $b_n=\frac{1}{2}\left(a_n-\sum\limits_{k=1}^{n-1}b_kb_{n-k}\right)$ is an even integer for $n\ge 1$, proving that $B(x)$ has integer coefficients.
\end{proof}

\section{Recurrence relations satisfied by some binomial sums}\label{sec 7}
In this section, we resolve conjectures on recurrence relations satisfied by the sequences $\mathrm{A127361}(n)$ (\cite{A127361}) and $\mathrm{A105872}(n)$ (\cite{A105872}).


We begin with the conjecture on the sequence $\mathrm{A127361}(n)$ (\cite{A127361}).

Let $a_n\coloneqq \sum\limits_{k=0}^n\binom{n}{\left\lfloor\frac{k}{2}\right\rfloor}\left(-2\right)^{n-k}$. In \cite{A127361}, Mathar conjectured the following result, which we prove.

\begin{thm}
For $n\ge 3$, we have
\[
2na_n+(5n-4)a_{n-1}-(8n-6)a_{n-2}-20(n-2)a_{n-3}=0.
\]
\end{thm}
\begin{proof}
For $n\ge 1$,
\begin{equation}\label{eqn:a_n}
\begin{aligned}
a_n &= \sum\limits_{k=0}^n\binom{n}{\left\lfloor\frac{k}{2}\right\rfloor}\left(-2\right)^{n-k}\\
&= \sum\limits_{k=0}^n\left(\binom{n-1}{\left\lfloor\frac{k}{2}\right\rfloor}+\binom{n-1}{\left\lfloor\frac{k}{2}\right\rfloor-1}\right)\left(-2\right)^{n-k}\\
&= \sum\limits_{k=0}^n\binom{n-1}{\left\lfloor\frac{k}{2}\right\rfloor}\left(-2\right)^{n-k}+\sum\limits_{k=0}^n\binom{n-1}{\left\lfloor\frac{k}{2}\right\rfloor-1}\left(-2\right)^{n-k}.
\end{aligned}
\end{equation}

Let
\begin{equation}\label{eqn:S_n}
S_n \coloneqq \sum\limits_{k=0}^n\binom{n-1}{\left\lfloor\frac{k}{2}\right\rfloor}\left(-2\right)^{n-k}
\end{equation}
and
\begin{equation}\label{eqn:S_n'}
S_n' \coloneqq \sum\limits_{k=0}^n\binom{n-1}{\left\lfloor\frac{k}{2}\right\rfloor-1}\left(-2\right)^{n-k}.
\end{equation}

Then
\begin{equation}\label{eqn:S_n-expanded}
S_n=\binom{n-1}{\left\lfloor \frac{n}{2}\right\rfloor}+(-2)\sum\limits_{k=0}^{n-1}\binom{n-1}{\left\lfloor\frac{k}{2}\right\rfloor}\left(-2\right)^{n-1-k} = \binom{n-1}{\left\lfloor \frac{n}{2}\right\rfloor} -2a_{n-1}   
\end{equation}

and
\begin{equation}\label{eqn:S_n'-expanded}
\begin{aligned}
S_n' &= \sum\limits_{k=2}^n\binom{n-1}{\left\lfloor\frac{k}{2}\right\rfloor-1}\left(-2\right)^{n-k}\\
&= \sum\limits_{j=0}^{n-2}\binom{n-1}{\left\lfloor\frac{j+2}{2}\right\rfloor-1}\left(-2\right)^{n-2-j}\\
&= \sum\limits_{j=0}^{n-2}\binom{n-1}{\left\lfloor\frac{j}{2}\right\rfloor}\left(-2\right)^{n-2-j}\\
&= -\frac{1}{2}\sum\limits_{j=0}^{n-2}\binom{n-1}{\left\lfloor\frac{j}{2}\right\rfloor}\left(-2\right)^{n-1-j}\\
&= -\frac{1}{2}\left(-\binom{n-1}{\left\lfloor \frac{n-1}{2} \right\rfloor}+\sum\limits_{j=0}^{n-1}\binom{n-1}{\left\lfloor\frac{j}{2}\right\rfloor}\left(-2\right)^{n-1-j}\right)\\
&= \frac{1}{2}\binom{n-1}{\left\lfloor \frac{n-1}{2} \right\rfloor} -\frac{a_{n-1}}{2}.
\end{aligned}
\end{equation}

From equations \eqref{eqn:a_n}, \eqref{eqn:S_n-expanded}, and \eqref{eqn:S_n'-expanded}, we obtain
\begin{equation}\label{eqn:a_n-expanded-1}
a_n=-\frac{5a_{n-1}}{2}+\frac{1}{2}\binom{n-1}{\left\lfloor \frac{n-1}{2} \right\rfloor}+\binom{n-1}{\left\lfloor \frac{n}{2} \right\rfloor}.
\end{equation}

We now want to simplify $\frac{1}{2}\binom{n-1}{\left\lfloor \frac{n-1}{2} \right\rfloor}+\binom{n-1}{\left\lfloor \frac{n}{2} \right\rfloor}$.

Suppose $n$ is even. Then $n=2m$ for some $m\ge 1$. In this case
\[
\frac{1}{2}\binom{n-1}{\left\lfloor \frac{n-1}{2} \right\rfloor}+\binom{n-1}{\left\lfloor \frac{n}{2} \right\rfloor} = \frac{1}{2}\binom{2m-1}{m-1}+\binom{2m-1}{m} = \frac{3}{2}\binom{2m-1}{m-1}=\frac{3}{2}\binom{n-1}{\left\lfloor \frac{n-1}{2} \right\rfloor}.
\]

Now suppose $n$ is odd. Then $n=2m+1$ for some $m\ge 0$. In this case
\[
\frac{1}{2}\binom{n-1}{\left\lfloor \frac{n-1}{2} \right\rfloor}+\binom{n-1}{\left\lfloor \frac{n}{2} \right\rfloor} = \frac{3}{2}\binom{2m}{m} = \frac{3}{2}\binom{n-1}{\left\lfloor \frac{n-1}{2} \right\rfloor}.
\]

Hence
\begin{equation}\label{eqn:floor-simplify}
\frac{1}{2}\binom{n-1}{\left\lfloor \frac{n-1}{2} \right\rfloor}+\binom{n-1}{\left\lfloor \frac{n}{2} \right\rfloor} = \frac{3}{2}\binom{n-1}{\left\lfloor \frac{n-1}{2} \right\rfloor}
\end{equation}
for all $n\ge 1$.

Equations \eqref{eqn:a_n-expanded-1} and \eqref{eqn:floor-simplify} give $a_n=-\frac{5a_{n-1}}{2}+\frac{3}{2}\binom{n-1}{\left\lfloor \frac{n-1}{2} \right\rfloor}$ for $n\ge 1$, which can be rewritten as
\begin{equation}\label{eqn:a_n-expanded-2}
2a_n+5a_{n-1}=3\binom{n-1}{\left\lfloor \frac{n-1}{2} \right\rfloor}.
\end{equation}

Define
\begin{equation}\label{eqn:b_n}
b_n\coloneqq \binom{n-1}{\left\lfloor \frac{n-1}{2} \right\rfloor}
\end{equation}

for $n\ge 1$, and write

\begin{equation}\label{eqn:B(x)}
B(x)\coloneqq \sum\limits_{n\ge 1}b_nx^n.
\end{equation}

We have $B(x)=B_1(x)+B_2(x)$, where $B_1(x)\coloneqq \sum\limits_{m\ge 0} b_{2m+1}x^{2m+1}$ and $B_2(x)\coloneqq \sum\limits_{m\ge 1} b_{2m}x^{2m}$.

Notice that $b_{2m+1}=\binom{2m}{m}$ for $m\ge 0$, and $b_{2m}=\binom{2m-1}{m-1}=\frac{(2m-1)!}{(m-1)!m!}=\frac{1}{2}\binom{2m}{m}$ for $m\ge 1$.

We therefore have $B_2(x)=\frac{1}{2}\sum\limits_{m\ge 1} \binom{2m}{m}x^{2m}$ and $B_1(x)=\sum\limits_{m\ge 0}\binom{2m}{m}x^{2m+1}=x+x\sum\limits_{m\ge 1}\binom{2m}{m}x^{2m}$.

By \cite[(2.5.11)]{gfology}, we know that 
\[\sum\limits_{m\ge 0}\binom{2m}{m}x^m=\frac{1}{\sqrt{1-4x}}.\]
Hence $\sum\limits_{m\ge 1}\binom{2m}{m}x^m=\frac{1}{\sqrt{1-4x}}-1$.

This implies that 
\[B(x)=B_1(x)+B_2(x)=\left(x+\frac{1}{2}\right)\left(\frac{1}{\sqrt{1-4x^2}}-1\right)+x=\frac{2x+1}{2\sqrt{1-4x^2}}-\frac{1}{2}.\]

Rewrite the above as
\begin{equation}\label{eqn:B(x)-expanded}
2B(x)+1=\frac{2x+1}{\sqrt{1-4x^2}}.
\end{equation}

Differentiating both sides of Equation \eqref{eqn:B(x)-expanded} with respect to $x$, we obtain

\[
2B'(x)=\frac{2}{\sqrt{1-4x^2}}+\frac{2x+1}{\left(1-4x^2\right)^{3/2}}\left(-\frac{1}{2}\right)(-8x) = \frac{2}{\sqrt{1-4x^2}}+\frac{8x^2+4x}{\left(1-4x^2\right)^{3/2}} = \frac{2(2x+1)}{\left(1-4x^2\right)^{3/2}}.
\]

The above can be rewritten as
\begin{equation}\label{eqn:B(x)-diff}
(1-4x^2)B'(x)=2B(x)+1.
\end{equation}

Differentiating both sides of Equation \eqref{eqn:B(x)} with respect to $x$, we get
\begin{equation}\label{eqn:B(x)-diff-alt}
B'(x)=\sum\limits_{n\ge 1}nb_n x^{n-1}.
\end{equation}

Combining equations \eqref{eqn:B(x)}, \eqref{eqn:B(x)-diff}, and \eqref{eqn:B(x)-diff-alt} gives
\[
(1-4x^2)\sum\limits_{n\ge 1}nb_n x^{n-1}-2\sum\limits_{n\ge 1}b_n x^n=1
\],

which can be rearranged to obtain
\begin{equation}\label{eqn:7.1-first-recur}
\sum\limits_{n\ge 1}nb_n x^{n-1} - 2\sum\limits_{n\ge 1}b_n x^n - \sum\limits_{n\ge 1}4nb_nx^{n+1} = 1.
\end{equation}

Let $n\ge 3$. Comparing the coefficients of $x^{n-1}$ on both sides of Equation \eqref{eqn:7.1-first-recur}, we obtain
\[
nb_n-2b_{n-1}-4(n-2)b_{n-2}=0.
\]

Multiplying both sides of the above by 3 yields
\[
n(3b_n)-2(3b_{n-1})-4(n-2)(3b_{n-2})=0.
\]

Substituting $3b_n=2a_n+5a_{n-1}$ (from equations \eqref{eqn:a_n-expanded-2} and \eqref{eqn:b_n}) gives us
\[
2na_n+(5n-4)a_{n-1}-(8n-6)a_{n-2}-20(n-2)a_{n-3}=0
\]
for $n\ge 3$, thus completing the proof.
\end{proof}


We now move on to the conjecture on sequence $\mathrm{A105872}(n)$ (\cite{A105872}). Note that we reuse some of the notation used in the previous part of this section after redefining the symbols appropriately.

Let $a_n\coloneqq \sum\limits_{k=0}^{\left\lfloor\frac{n}{3}\right\rfloor}\binom{2n-3k}{n}$ for $n\ge 0$. In \cite{A105872}, Mathar conjectured the following result, which we prove.

\begin{thm}
For $n\ge 3$, we have
\[
3(n+1)(7n-2)a_n -6(7n+5)(2n-1)a_{n-1} +(n+1)(7n-2)a_{n-2} -2(7n+5)(2n-1)a_{n-3}=0
\]
\end{thm}
\begin{proof}
Let
\[A(x)\coloneqq \sum\limits_{n\ge 0}a_nx^n=\sum\limits_{n\ge 0}\sum\limits_{k=0}^{\left\lfloor\frac{n}{3}\right\rfloor}\binom{2n-3k}{n}x^n.\]

Let $m=n-3k$.

Then
\begin{equation}\label{eqn:A(x)-2}
A(x)=\sum\limits_{n\ge 0}\sum\limits_{k=0}^{\left\lfloor\frac{n}{3}\right\rfloor}\binom{2n-3k}{n}x^n=\sum\limits_{k\ge 0}\sum\limits_{m\ge 0}\binom{2m+3k}{m+3k}x^{m+3k}=\sum\limits_{k\ge 0}x^{3k}\sum\limits_{m\ge 0}\binom{2m+3k}{m}x^m.
\end{equation}

By \cite[(2.5.15)]{gfology}, we have 
\[
\sum\limits_{m\ge 0}\binom{2m+3k}{m}x^m=\frac{1}{\sqrt{1-4x}}\left(\frac{1-\sqrt{1-4x}}{2x}\right)^{3k}.
\]

Let $u\coloneqq \sqrt{1-4x}$ and $y\coloneqq \frac{1-u}{2x}$.

Making the above substitutions in Equation \eqref{eqn:A(x)-2} yields 
\begin{equation}\label{eqn:A(x)-3}
A(x)=\sum\limits_{k\ge 0} x^{3k}\frac{1}{\sqrt{1-4x}}\left(\frac{1-\sqrt{1-4x}}{2x}\right)^{3k} = \sum\limits_{k\ge 0}x^{3k}\frac{y^{3k}}{u} = \frac{1}{u}\sum\limits_{k\ge 0} \left(x^3y^3\right)^k = \frac{1}{u(1-x^3y^3)}.
\end{equation}

We have that $y=\frac{1-\sqrt{1-4x}}{2x}$. This implies that $xy^2-y+1=0$.

Using this we obtain $x^3y^3=x^2yxy^2=x^2y(y-1)=x^2y^2-x^2y=x(y-1)-x^2y$, and therefore
\[
1-x^3y^3 = 1+x-x(1-x)y = 1+x-x(1-x)\frac{(1-u)}{2x} = \frac{1+3x}{2}+\frac{(1-x)u}{2}.
\]

Substituting this in Equation \eqref{eqn:A(x)-3} gives
\begingroup
\allowdisplaybreaks
\begin{align}\label{eqn:A(x)-4}
A(x) &= \frac{1}{u\left(\frac{1+3x}{2}+\frac{(1-x)u}{2}\right)}\\
&= \frac{2}{(1-x)u^2+(1+3x)u}\notag\\
&= \frac{2}{(1-x)(1-4x)+(1+3x)u}\notag\\
&= \frac{2((1-x)(1-4x)-(1+3x)u)}{((1-x)(1-4x)+(1+3x)u)((1-x)(1-4x)-(1+3x)u)}\notag\\
&= \frac{2(1-x)(1-4x)-2(1+3x)u}{\left(1-x\right)^2\left(1-4x\right)^2-\left(1+3x\right)^2(1-4x)}\notag\\
&= \frac{2(1-x)(1-4x)-2(1+3x)u}{-4x(1-4x)(3+x^2)}\notag\\
&= \frac{1+3x}{2(3+x^2)u} - \frac{(1-x)}{2x(3+x^2)}.\notag
\end{align}
\endgroup

Multiplying by $3+x^2$, we obtain
\begin{equation}\label{eqn:A(x)-5}
(3+x^2)A(x) = \frac{1+3x}{2x\sqrt{1-4x}} - \frac{1}{2x} + \frac{1}{2}.  
\end{equation}

Let $n\ge 3$. We now compare the coefficients of $x^n$ on both sides of Equation \eqref{eqn:A(x)-5}. We have
\begin{equation}\label{eqn:coeff-1}
[x^n] (3+x^2)A(x)=3a_n+a_{n-2}
\end{equation}
and
\begin{equation}\label{eqn:coeff-2}
\begin{aligned}
[x^n]\left(\frac{1}{2x\sqrt{1-4x}} + \frac{3}{2\sqrt{1-4x}} - \frac{1}{2x} + \frac{1}{2}\right) &= [x^n]\left(\frac{1}{2x\sqrt{1-4x}} + \frac{3}{2\sqrt{1-4x}}\right)\\
&= \frac{1}{2}\binom{2n+2}{n+1}+\frac{3}{2}\binom{2n}{n},
\end{aligned}
\end{equation}
where the last equality follows from \cite[(2.5.11)]{gfology}.

Simplifying further yields
\begin{align*}
\frac{1}{2}\binom{2n+2}{n+1}+\frac{3}{2}\binom{2n}{n} &= \frac{1}{2}\frac{2(2n+1)((2n)!)}{(n+1)(n!n!)} + \frac{3}{2}\binom{2n}{n}\\
&= \left(\frac{2n+1}{n+1}+\frac{3}{2}\right)\binom{2n}{n}\\
&= \frac{7n+5}{2(n+1)}\binom{2n}{n}.
\end{align*}

Thus, for $n\ge 3$,
\begin{equation}\label{eqn:7.2-first-recur}
3a_n+a_{n-2} = \frac{7n+5}{2(n+1)}\binom{2n}{n}.
\end{equation}

Let
\begin{equation}\label{eqn:7.2-b_n}
b_n \coloneqq \frac{7n+5}{2(n+1)}\binom{2n}{n}.
\end{equation}

Then, for $n\ge 3$,
\begin{equation}\label{eqn:7.2-b_n-2}
\frac{b_n}{b_{n-1}} = \frac{\frac{7n+5}{2(n+1)}\binom{2n}{n}}{\frac{7n-2}{2n}\binom{2n-2}{n-1}} = \left(\frac{7n+5}{n+1}\right)\left(\frac{n}{7n-2}\right)\left(\frac{2(2n-1)}{n}\right) = \frac{2(2n-1)(7n+5)}{(n+1)(7n-2)}.
\end{equation}

By Equation \eqref{eqn:7.2-b_n-2}, for $n\ge 3$, we obtain
\begin{equation}\label{eqn:7.2-b_n-recur}
(n+1)(7n-2)b_n-2(2n-1)(7n+5)b_{n-1}=0.
\end{equation}

Using equations \eqref{eqn:7.2-first-recur} and \eqref{eqn:7.2-b_n} in Equation \eqref{eqn:7.2-b_n-recur} gives
\[
(n+1)(7n-2)(3a_n+a_{n-2})-2(2n-1)(7n+5)(3a_{n-1}+a_{n-3})=0
\]
for $n\ge 3$, which simplifies to
\[
3(n+1)(7n-2)a_n -6(7n+5)(2n-1)a_{n-1} +(n+1)(7n-2)a_{n-2} -2(7n+5)(2n-1)a_{n-3}=0,
\]

thus completing the proof.
\end{proof}

\section{An asymptotic relation satisfied by a binomial coefficient sum}\label{sec 8}

Using notation from \cite{ECombi}, for $f,g:\mathbb{Z}_{\ge 0}\to\mathbb{C}$, we write $f\sim g$ if $\lim\limits_{n\to\infty} \frac{f(n)}{g(n)}=1$.


In \cite{A094213}, McEachen conjectured that $\sum\limits_{k=0}^n\binom{9n}{9k}\sim \frac{1}{9}\exp(9n\log(2))=\frac{2^{9n}}{9}$. We prove a generalisation of this conjecture.

\begin{thm}
For any positive integer $m$, we have $\sum\limits_{k=0}^n\binom{mn}{mk}\sim \frac{1}{m}\exp(mn\log(2))=\frac{2^{mn}}{m}$.
\end{thm}
\begin{proof}
For $m\in\mathbb{N}$ and $n\in\mathbb{N}_0$, define the following:
\begin{enumerate}
\item $f(m,n)\coloneqq \sum\limits_{k=0}^n\binom{mn}{mk}$.
\item $g_{m,n}(x)\coloneqq (1+x)^{mn}=\sum\limits_{r=0}^{mn}\binom{mn}{r}x^r$.
\item $h_{m,n}(x)\coloneqq \sum\limits_{r=0}^{n}\binom{mn}{mr}x^{mr}$.
\end{enumerate}

Then $f(m,n)=h_{m,n}(1)$.

Let $\alpha_m\coloneqq \exp(\frac{2\pi i}{m})$.

By \cite[Equation~(6)]{Multisection}, we have 
\[h_{m,n}(x)=\frac{1}{m}\sum\limits_{s=0}^{m-1}g_{m,n}(\alpha_m^sx)=\frac{1}{m}\sum\limits_{s=0}^{m-1}\left(1+\alpha_m^sx\right)^{mn}.
\]

This implies that $f(m,n)=\frac{1}{m}\sum\limits_{s=0}^{m-1}\left(1+\alpha_m^s\right)^{mn}$.

We want to evaluate $\lim\limits_{n\to\infty} \frac{mf(m,n)}{2^{mn}}$.

We have
\begin{align*}
\lim\limits_{n\to\infty} \frac{mf(m,n)}{2^{mn}} &= \lim\limits_{n\to\infty} \frac{\sum\limits_{s=0}^{m-1}\left(1+\alpha_m^s\right)^{mn}}{2^{mn}}\\
&= \lim\limits_{n\to\infty} \frac{2^{mn}+\sum\limits_{s=1}^{m-1}\left(1+\alpha_m^s\right)^{mn}}{2^{mn}}\\
&= 1+\lim\limits_{n\to\infty} \frac{\sum\limits_{s=1}^{m-1}\left(1+\alpha_m^s\right)^{mn}}{2^{mn}}.
\end{align*}

We also have
\[
\left|\lim\limits_{n\to\infty} \frac{\sum\limits_{s=1}^{m-1}\left(1+\alpha_m^s\right)^{mn}}{2^{mn}}\right| = \lim\limits_{n\to\infty} \frac{\left|\sum\limits_{s=1}^{m-1}\left(1+\alpha_m^s\right)^{mn}\right|}{2^{mn}} \le \lim\limits_{n\to\infty} \frac{\sum\limits_{s=1}^{m-1}\left|1+\alpha_m^s\right|^{mn}}{2^{mn}}
\]

and

\[|1+\alpha_m^s|=\left|1+\cos\left(\frac{2\pi s}{m}\right)+i\sin\left(\frac{2\pi s}{m}\right)\right|=\sqrt{\left(1+\cos\left(\frac{2\pi s}{m}\right)\right)^2+\left(\sin\left(\frac{2\pi s}{m}\right)\right)^2}\]

\[=\sqrt{1+\left(\cos\left(\frac{2\pi s}{m}\right)\right)^2+\left(\sin\left(\frac{2\pi s}{m}\right)\right)^2+2\cos\left(\frac{2\pi s}{m}\right)}=\sqrt{2\left(1+\cos\left(\frac{2\pi s}{m}\right)\right)}.
\]

This shows that $|1+\alpha_m^s|<2$ for $s\in\{1,\dots,m-1\}$.

Let $R_m=\max\{|1+\alpha_m^s|:1\le s\le m-1\}$. Then $0<R_m<2$.

This gives

\[
0\le \lim\limits_{n\to\infty} \frac{\sum\limits_{s=1}^{m-1}\left|1+\alpha_m^s\right|^{mn}}{2^{mn}} \le \lim\limits_{n\to\infty} \frac{(m-1)R_m^{mn}}{2^{mn}}= \lim\limits_{n\to\infty} \frac{m-1}{\left(\frac{2}{R_m}\right)^{mn}}=0,\text{ since $\frac{2}{R_m}>1$.}
\]

Thus $\lim\limits_{n\to\infty} \frac{mf(m,n)}{2^{mn}}=1+\lim\limits_{n\to\infty} \frac{\sum\limits_{s=1}^{m-1}\left(1+\alpha_m^s\right)^{mn}}{2^{mn}}=1$, thereby completing the proof.
\end{proof}

\printbibliography

\end{document}